\documentclass[12pt,reqno]{amsart}
\usepackage{amsmath}
\usepackage{amssymb}
\usepackage{graphicx}
\usepackage{tikz-cd}
\usepackage[T1]{fontenc}
\usepackage[utf8]{inputenc}

\newcommand{\Sp}{\operatorname{Sp}}

\usepackage[colorlinks=true,
            citecolor={teal},
            linkcolor=blue,
            urlcolor=blue]{hyperref}

\usepackage[all]{xy}
\usepackage{xypic}
\usepackage{mathtools}

\usepackage{xspace}
\usepackage{bm}
\usepackage{amsmath}
\usepackage{amstext}
\usepackage{amsfonts}
\usepackage[mathscr]{euscript}
\usepackage{amscd}
\usepackage{latexsym}
\usepackage{amssymb}
\usepackage{enumerate}

\usepackage{booktabs}
\usepackage{tabularx}
\usepackage{multirow}
\usepackage{amsmath}

\usepackage{anysize}
\marginsize{2.5cm}{2.5cm}{2.5cm}{2.5cm}
\usepackage{setspace}

\theoremstyle{plain}
    \newtheorem{theorem}{Theorem}[section]
    \newtheorem{proposition}[theorem]{Proposition}
    \newtheorem{lemma}[theorem]{Lemma}
    
    \newtheorem{conjecture}[theorem]{Conjecture}

    \newtheorem{subsec}[theorem]{}
    
    \newtheorem*{thma}{Theorem A}
    \newtheorem*{thmb}{Theorem B}

    \newtheorem*{conj}{Conjecture}

\theoremstyle{definition}
    \newtheorem{construction}[theorem]{Construction}

\theoremstyle{remark}
        \newtheorem{remark}[theorem]{Remark}

\renewcommand{\thefigure}{\arabic{section}.\arabic{figure}}

\newenvironment{myeq}[1][]
{\stepcounter{figure}\begin{equation}\tag{\thefigure}{#1}}
{\end{equation}}

\newcommand{\eat}[1]{}

\allowdisplaybreaks

\title[Degrees of Maps and Cohomological Rigidity of Partial Flag Manifolds]{Degrees of Maps and Cohomological Rigidity of Partial Flag Manifolds}
\author[Manas Mandal]{Manas Mandal}
\address{Department of Mathematics, IIT Kanpur, Kanpur 208016, India}
\email{manas.imsc@gmail.com}

\thanks{ }

\date{\today}                                          

\begin{document}

\begin{abstract}
We study the existence of continuous maps with nonzero Brouwer degree between partial flag manifolds. We prove that every continuous map between distinct complex or quaternionic partial flag manifolds has degree zero if either the domain or the codomain is not a Grassmannian. This establishes, for such partial flag manifolds, the analogue of the results of Ramani--Sankaran and Sankaran--Sarkar for complex and quaternionic Grassmannians, as well as an algebraic-geometric result of Paranjape and Srinivas for complex Grassmannians. As a consequence, we prove that complex and quaternionic partial flag manifolds are cohomologically rigid; that is, their rational cohomology rings determine their homeomorphism types. We conjecture that complex and quaternionic partial flag manifolds are homologically rigid.
\end{abstract}

\maketitle

\section{Introduction}

The existence of continuous maps of nonzero Brouwer degree between manifolds is a classical topic in topology. Such maps relate the topology of the domain and codomain in a strong way, inducing monomorphisms on rational cohomology and imposing significant constraints on their topology. When the rational cohomology rings of manifolds within a given class exhibit sufficient algebraic differences, these monomorphisms are ruled out, yielding non-existence results for maps of non-zero degree. Consequently, this leads to the fundamental question on cohomological rigidity: to what extent do cohomology algebras classify or determine manifolds in a given class?

In this paper, we apply this approach to complex and quaternionic partial flag manifolds, exploiting the structure of their rational cohomology, particularly in degree two, to study both nonzero-degree maps and cohomological rigidity.

In the algebraic-geometric setting, Paranjape and Srinivas \cite{PS89} proved that a finite surjective morphism between two complex Grassmannians can exist only if the two Grassmannians are isomorphic, unless the codomain is a projective space. In other words, there is no finite surjective morphism between two such distinct complex Grassmannians. In topological setting, Ramani and Sankaran \cite{RS97} initiated the study of nonzero degree maps between Grassmannians. They proved the nonexistence of nonzero degree maps between distinct complex (resp., quaternionic) Grassmannians in roughly half of the cases and conjectured that the same conclusion holds in the remaining cases. Subsequently, Sankaran and Sarkar \cite{SS09} settled the conjecture under an additional hypothesis. We refer to Table~\ref{tab:cases_classification} for a summary of the related developments. There have also been several other studies related to the existence and nonexistence of maps between various classes of manifolds; see, for example, \cite{Bar17, BS18, Can03, CS14, CM18, FR25, Kum25, Ram09, Sch99, SS09}.

The natural next step is to investigate the existence of nonzero degree maps between partial flag manifolds beyond the Grassmannian case. More precisely, we ask whether there exists a continuous map of nonzero degree between two distinct complex (or quaternionic) partial flag manifolds. We answer this question by proving Theorem A.

Let $\mathbb{F}$ denote either the field of complex numbers $\mathbb{C}$ or the skew field of quaternions $\mathbb{H}$, and let
$\nu=(n_1,\ldots,n_r)$ be a sequence of positive integers with
$n=n_1+\cdots+n_r$. Then the $\mathbb{F}$-partial flag manifold
$\mathbb{F}G(\nu)$ is the space of all orthogonal decompositions
$\mathbb{F}^n=V_1\oplus\cdots\oplus V_r$, where
$\dim_{\mathbb{F}}V_i=n_i$ for each $1\le i\le r$. The real dimension of $\mathbb FG(\nu)$ is given by $d\cdot\sum _{i<j}n_in_j$, where d is the real dimension of $\mathbb F$.

\begin{thma}[Theorem~\ref{main thm}, Proposition~\ref{prop r>2 s=2}]\label{intro: main thm}
    Let $f\colon \mathbb{F}G(\nu)\to\mathbb{F}G(\mu)$ be a continuous map between two partial flag manifolds of equal dimension, where $\nu=(n_1,\ldots,n_r)$ and $\mu=(m_1,\ldots,m_s)$ are not permutations of each other. Assume that $(r,s)\neq(2,2)$. Then the Brouwer degree of $f$ is zero.
\end{thma}

As an immediate consequence, Theorem~A shows that the existence of a finite surjective morphism between two complex partial flag varieties, at least one of which is not a Grassmannian, determines their isomorphism type. Indeed, such a morphism necessarily has nonzero degree, and hence Theorem~A forces the domain and codomain to be isomorphic. This provides the analogue for such complex partial flag manifolds, of the result of Paranjape and Srinivas \cite[Proposition~6]{PS89} for complex Grassmannians.

An important consequence of Theorem~A is the following \textit{cohomological rigidity} theorem for complex and quaternionic partial flag manifolds. Determining the extent to which the cohomology algebra determines the underlying space, a phenomenon known as cohomological rigidity, is an important and active area of research. Such cohomological rigidity results for various classes of spaces have been established in \cite{CMS10, PS16, KS15, SS21,  HK22, HKM22, BCP25, CHJ25,  HKT26}.

\begin{thmb}[Theorem~\ref{prop:cohomological-rigidity}]
    Let
    $M=\mathbb{F}G(n_1,\ldots,n_r)$ and
    $N=\mathbb{F}G(m_1,\ldots,m_s)$
    be partial flag manifolds. If the cohomology algebras
    $H^*(M;\mathbb{Q})\cong H^*(N;\mathbb{Q})$
    as graded $\mathbb{Q}$-algebras, then $M$ and $N$ are homeomorphic.
\end{thmb}

The proof of Theorem A is based on the fact that a continuous map of nonzero degree induces a monomorphism in rational cohomology. In the setting of complex partial flag manifolds, the induced map on the second rational cohomology preserves the complement of the braid hyperplane arrangement. This preservation property imposes strong rigidity, forcing the induced monomorphism on the cohomology algebra to be an automorphism. The proof is then completed by exploiting the heights of degree-two cohomology classes and the description of Kähler classes in the cohomology algebra of complex partial flag manifolds studied in \cite{BHH83}. 

The quaternionic case of Theorem A follows from the proof for the complex case. This is possible because the rational cohomology algebra of a quaternionic partial flag manifold has the same algebraic presentation as that of the corresponding complex partial flag manifold, with the only difference being that the generators occur in twice the degree.

Theorem B follows by combining Theorem A with the realization theorem of \cite[Theorem 1.2]{GH81}, which ensures the realization of suitable isomorphisms between rational cohomology algebras by continuous maps.

Finally, we propose the following conjecture on the \emph{homological rigidity} of complex and quaternionic partial flag manifolds, asserting that the assumption of an isomorphism of cohomology algebras in Theorem~B can be weakened to isomorphisms of all cohomology groups, or equivalently, isomorphisms of all homology groups.

\begin{conj}
    Let $\mathbb{F}G(\nu)$ and $\mathbb{F}G(\mu)$ be partial flag manifolds. If the homology groups $H_i\bigl(\mathbb{F}G(\nu);\mathbb{Z}\bigr)
    \cong
    H_i\bigl(\mathbb{F}G(\mu);\mathbb{Z}\bigr)$
    for every $i$, then $\mathbb{F}G(\nu)$ and $\mathbb{F}G(\mu)$ are homeomorphic.
\end{conj}

We verify Conjecture for Grassmannians and almost complete flag manifolds.

\section{Preliminaries}

\subsection{Partial flag manifolds}
Let $\nu=(n_1,\ldots,n_r)$ be a nontrivial partition of $n$. We denote by $\mathbb{C}G(\nu)$ the partial complex flag manifold of type $\nu$, that is, the homogeneous space
$\mathbb{C}G(\nu)=U(n)/(U(n_1)\times\cdots\times U(n_r))$. Equivalently, it is the space of all orthogonal decompositions $\mathbb{C}^n=V_1\oplus\cdots\oplus V_r$ such that $\dim_{\mathbb C}V_i=n_i$ for each $i$.
The complex dimension of $\mathbb CG(\nu)$ is given by $\sum_{ i<j}n_in_j.$

For $j=1,\ldots,r$, let $\gamma_j$ be the canonical complex vector bundle of rank $n_j$ over $\mathbb{C}G(\nu)$, whose fibre at
$V_1\oplus\cdots\oplus V_r\in\mathbb{C}G(\nu)$ is $V_j$. Since
$\mathbb{C}^n=V_1\oplus\cdots\oplus V_r$ at every point of $\mathbb{C}G(\nu)$, we have
$\gamma_1\oplus\cdots\oplus\gamma_r \cong n\varepsilon_{\mathbb C},$
where $\varepsilon_{\mathbb C}$ denotes the trivial complex line bundle. Taking total Chern classes on both sides of the bundle isomorphism, one obtains
\begin{equation}
    \prod_{1 \leq j \leq r}\big  
    (1+c_1(\gamma_j)+c_2(\gamma_j)+\cdots+c_{n_j}(\gamma_j)\big) = 1.
\end{equation}

The cohomology ring $ H^*\big(\mathbb{C} G(\nu); \mathbb{Z}\big) $ is given by
\[
H^*\big(\mathbb{C} G(\nu); \mathbb{Z}\big) \;\cong\; \mathbb{Z}[c_{i,j}: 1 \leq i \leq n_j, 1 \leq j \leq r] \Big/ \langle h_1, \dots, h_n \rangle,
\]
where the elements $ h_j $ are defined by the equation
\[
\sum_{0 \leq i \leq n} h_i t^i = \prod_{1 \leq j \leq s} \left( 1 + c_{1,j} t + \cdots + c_{n_j,j} t^{n_j} \right),
\]
and   $ c_{i,j} $ corresponds to $ c_i(\gamma_j) $ under the isomorphism.

For the complex Grassmannian $\mathbb CG(n_1,n_2)$, the cohomology algebra 
\begin{equation}\label{cohom c_i bar c_i}
    H^*\big(\mathbb CG(n_1,n_2);\mathbb Z\big)
    \cong
    \mathbb Z[c_1,\ldots,c_{n_1},\bar c_1,\ldots,\bar c_{n_2}]
    \Big/
    \langle h_1,\ldots,h_{n_1+n_2}\rangle,
\end{equation}
where $c_i$ and $\bar c_i$ corresponds to the $i$-th Chern classes of the canonical complex vector bundles $\gamma_1$ and $\gamma_2$, respectively. The relations are determined by the identity
\[
h_i:=\sum_{p+q=i}c_p\bar c_q=0,\qquad
1\le i\le n_1+n_2,
\]
where $c_0=\bar c_0=1$ and $c_p=\bar c_q=0$ whenever the indices are out of range.

Since the coefficient of $\bar c_i$ in the relation $h_i=0$ is $1$ for $1\le i\le n_2$, these relations can be solved inductively to express $\bar c_1,\ldots,\bar c_{n_2}$ as polynomials in $c_1,\ldots,c_{n_1}$. Substituting these expressions into the remaining relations $h_{n_2+1},\ldots,h_{n_1+n_2}$ yields an equivalent presentation involving only the generators $c_1,\ldots,c_{n_1}$:
\begin{equation}\label{cohom c_i}
    H^*\big(\mathbb CG(n_1,n_2);\mathbb Z\big)
    \cong
    \mathbb Z[c_1,\ldots,c_{n_1}]
    \Big/
    \langle h_{n_2+1},\ldots,h_{n_1+n_2}\rangle,
\end{equation}
where each $h_i$ is understood to be written entirely in terms of $c_1,\ldots,c_{n_1}$ after eliminating the variables $\bar c_1,\ldots,\bar c_{n_2}$.

The quaternionic partial flag manifold of type $\nu=(n_1,\ldots,n_r)$ is the homogeneous space
$\mathbb HG(\nu):=\Sp(n)/\Sp(n_1)\times\cdots\times\Sp(n_r).$ As in the complex case, there are canonical quaternionic vector bundles $\eta_i$, $1\le i\le r$, of quaternionic rank $n_i$.

The integral cohomology of $\mathbb HG(\nu)$ admits a presentation analogous to that of the complex partial flag manifold:
\begin{equation}\label{cohomology quaternionic flag}
    H^*\big(\mathbb HG(\nu);\mathbb Z\big)
    \cong
    \mathbb Z\big[q_{i,j}:1\le i\le n_j,\;1\le j\le r\big]\Big/\mathcal I_{\mathbb H},
\end{equation}
where $\mathcal I_{\mathbb H}$ is generated by the positive-degree homogeneous components of the identity
$$\prod_{j=1}^{r}(1+q_{1,j}+\cdots+q_{n_j,j})=1, \quad \text{where $\deg q_{i,j}=4i.$}$$

Equivalently, the presentation \eqref{cohomology quaternionic flag} is obtained from that of the complex partial flag manifold by replacing each generator $c_{i,j}$ with $q_{i,j}=c_{2j}(\eta_i)$, thereby doubling the degree of every generator while leaving the defining polynomial relations unchanged. Consequently, the cohomology algebra of $\mathbb HG(\nu)$ may be viewed as the ``degree-doubled'' analogue of the cohomology algebra of $\mathbb CG(\nu)$, a fact that will be used repeatedly in this article.

\subsection{Degree of a map}

Let $M$ and $N$ be connected, closed, orientable manifolds of the same
dimension $d$, and let $f\colon M\to N$ be a continuous map. The
\emph{Brouwer degree} of $f$, denoted by $\deg(f)$, is the unique integer
determined by
\[
f^*([N])=\deg(f)\,[M],
\]
where $[M]\in H^d(M;\mathbb{Z})$ and
$[N]\in H^d(N;\mathbb{Z})$ are the fundamental cohomology classes.

The following well-known facts will be used repeatedly.

\begin{lemma}\label{degree-lemma}
    Let $M$ and $N$ be connected, closed, orientable manifolds of equal dimension, and let $f\colon M\to N$ be a continuous map. Then
    \[
    \deg(f)\neq 0
    \quad\Longleftrightarrow\quad
    f^*\colon H^*(N;\mathbb{Q})\to H^*(M;\mathbb{Q})
    \text{ is injective}.
    \]
\end{lemma}

We refer the reader to \cite{MS74} for characteristic classes and to \cite{Spa66} for the Brouwer degree and other basic notions in algebraic topology.

\section{Degrees of maps between partial flag manifolds}

In this section, we study continuous maps of nonzero Brouwer degree between partial flag manifolds over $\mathbb{F}\in{\mathbb{C},\mathbb{H}}$. Our main result shows that any map between two distinct partial flag manifolds of the same dimension, one of which is not a Grassmannian, has degree zero. The proof combines the fact that a map of nonzero degree induces an injective homomorphism on rational cohomology with the observation that, in our setting, such maps preserve the complements of the corresponding braid hyperplane arrangements.

Let $\mathbb{C}G(\nu)$ be the complex partial flag manifold of type
$\nu=(n_1,n_2,\ldots,n_r)$. For each $i$, let
$x_i\in H^2(\mathbb{C}G(\nu);\mathbb{Z})$ denote the first Chern class of the canonical complex vector bundle of rank $n_i$ over $\mathbb{C}G(\nu)$. 

Now we begin by recalling the following results from \cite{BHH83}.

\begin{theorem}[Theorem 2.3, Corollary 3.2, \cite{BHH83}]\label{Kahler class description}
    A cohomology class of the form
    \[
    \sum_{i=1}^{\ell} a_i x_i \in H^2(\mathbb{C}G(\nu);\mathbb{Z})
    \]
    is a K\"ahler class if and only if the integers
    $a_1,a_2,\ldots,a_\ell$ are all distinct.
\end{theorem}

\begin{theorem}[Theorem 3.1, \cite{BHH83}]\label{heights of deg 2 elements}
    Consider a cohomology class of the form
    \[
    x=\sum_{i=1}^{\ell} a_i x_i \in H^2(\mathbb{C}G(\nu);\mathbb{Z}).
    \]
    Let $\{b_1,b_2,\ldots,b_t\}$ denote the set of distinct values among the integers
    $a_1,a_2,\ldots,a_\ell$. For each $1\le j\le t$, define $m_j=\sum_{a_i=b_j} n_i.$
    Then the height of $x$ is given by
    \[
    h(x)=\sum_{1\le i<j\le t} m_i m_j.
    \]
\end{theorem}

We now prove a series of lemmas that prepare the way for the proof of the main theorem.

\begin{lemma}\label{ninj=mimj}
    Let $\nu = (n_1, n_2, \ldots, n_r)$ and $\mu = (m_1, m_2, \ldots, m_r)$ be two finite sequences of natural numbers such that $n_i n_j = m_i m_j$ for all $1 \le i < j \le r$. If $r \ge 3$, then $n_i = m_i$ for all $i = 1, 2, \ldots, r$.
\end{lemma}

\begin{proof}
    Let $i, j, k$ be three distinct indices in $\{1, 2, \ldots, r\}$. By the given hypothesis, we have the following system of equations:
    \begin{align}
        n_i n_j &= m_i m_j \label{eq1} \\
        n_j n_k &= m_j m_k \label{eq2} \\
        n_i n_k &= m_i m_k \label{eq3}
    \end{align}
    
    Multiplying equations \eqref{eq1} and \eqref{eq3} yields: $n_i^2 (n_j n_k) = m_i^2 (m_j m_k)$.
    Using \eqref{eq2} in the expression gives $n_i^2 (m_j m_k) = m_i^2 (m_j m_k).$
     Dividing both sides by $m_j m_k$, we get $n_i^2 = m_i^2$ which implies $n_i=m_i.$
     Since the choice of  $i$ was arbitrary, $n_i=m_i$ holds for all $i$.
\end{proof}

\begin{lemma}\label{row - row is scaling and shift}
    Let $A=[R_1,\ldots,R_n]^T$ be a real $n\times m$ matrix of rank $m\ge 3$. Suppose that for every $a\neq b$ there exist $i\neq j$ and $\lambda_{a,b}> 0$ such that
    \[
    R_a-R_b=\lambda_{a,b}(e_i-e_j),
    \]  where $\{e_1,\ldots , e_n\}$ denotes the standard basis of $\mathbb R^n$.
    Then $m=n$ and there exist $R\in\mathbb R^m$, $\lambda\in \mathbb R\setminus\{0\}$, and
    $k\in\{1,\ldots,m\}$ such that the set of all rows
    \[
    \{R_1,\ldots,R_m\}\subseteq
    \{\,R+\lambda(e_k-e_t):1\le t\le m\,\}.
    \]
\end{lemma}

\begin{proof}
    Consider the complete directed graph $G$ on the vertex set $\{R_1,\ldots,R_n\}\subset \mathbb R^m$, where each directed edge joining $R_b$ and $R_a$ $(b\neq a)$ is identified with the vector $R_a-R_b$. 

    We need the following observation.\\
    \textbf{Claim 1:} The scalars corresponding to the edges incident to $R_1$ are all equal; that is, $\lambda_{a,1}=\lambda_{b,1}$ for every $a,b\neq 1$.
    
    First we show that $\lambda_{a,1}=\lambda_{b,1}$ for any two distinct edges $R_a-R_1, R_b-R_1$ which are not scalar multiple of each other. By hypothesis, $R_a-R_1=\lambda_{a,1}(e_u-e_v)$ and $R_b-R_1=\lambda_{b,1}(e_p-e_q)$ for some $u\neq v, p\neq q$ and $\lambda_{a,1},\lambda_{b,1}>0.$ Then 
    \begin{equation}\label{mu(ei-ej)}
        R_a-R_b=\lambda_{a,1}(e_u-e_v)-\lambda_{b,1}(e_p-e_q)=\mu (e_i-e_j), \text{ for some $\mu>0$ and $i\neq j.$}
    \end{equation}
    Thus $\{u,v\}\cap \{p,q\}\neq \emptyset.$ Since $R_a-R_1, R_b-R_1$ are not scalar multiple of each other, we have $\{u,v\}\neq \{p,q\}.$ Now the possible cases are
    \[
    (i)\; u=p,v\neq q,\quad (ii)\;u=q,v\neq p,\quad (iii)\;v=p, u\neq q,\quad \text{and }(iv)\; v=q, u\neq p.
    \]
    The cases $(i)$ and $(iv)$ implies that $\lambda_{a,1}=\lambda_{b,1}$ and $(ii)$ and $(iii)$ are impossible, using \eqref{mu(ei-ej)}.

    Since $\mathrm{rank} (A)=m\ge 3,$ there exist at least two edges, say $R_g-R_1$ and $R_h-R_1$, which are not scalar multiples of each other. In fact, if all the edges are scalar multiples of each other,  there exist a fixed pair $i\ne j$ such that for each $a\neq 1$,
    \[
    R_a-R_1=\lambda_{a,1}(e_i-e_j) \quad \implies\quad  R_a=R_1+\lambda_{a,1}(e_i-e_j),\quad \text{ for some }\lambda_{a,1}\in \mathbb R.
    \]
    This shows that every row of $A$ belongs to the subspace $\mathrm{span}\{R_1,e_i-e_j\}\subseteq \mathbb R^m$, contradicting to the fact that $\mathrm{rank}(A)=m\ge 3.$

    Now we consider an edge $R_c-R_1$ which is a scalar multiple of $R_d-R_1$. Observe that $R_d-R_1$ cannot be a scalar multiple of both $R_g-R_1$ and $R_h-R_1$; without loss of generality, assume that it is not a scalar multiple of $R_g-R_1$. By the preceding argument, it follows that $\lambda_{d,1}=\lambda_{g,1}$. Since $R_c-R_1$ is a scalar multiple of $R_d-R_1$, it is also not a scalar multiple of $R_g-R_1$. Applying the same argument once again, we obtain $\lambda_{c,1}=\lambda_{g,1}$. Hence, $\lambda_{c,1}=\lambda_{d,1}$. This proves Claim 1.

    We set, using Claim 1, that $R_a-R_1=\lambda(e_{i_a}-e_{j_a}), \;\lambda>0$ for each $a\neq 1$.

    In order to complete the proof of the lemma, we need to show one more claim. \\
    \textbf{Claim 2:} Either $i_a=i_b$, or $j_a=j_b$, for all $a, b\neq 1.$

    We already have seen that $\{i_a,j_a\}\cap \{i_b,j_b\}\neq \emptyset$ for $a,b \neq 1$, otherwise $R_a-R_b$ will have at least four nonzero coordinates and cannot be of the form $\lambda_{a,b}(e_i-e_j)$.

    If possible let there be three \textit{distinct} edges $R_a-R_1, R_b-R_1, R_c-R_1\neq 0$ in $\mathbb R^m$ be such that $i_a=i_b$ and $j_a=j_c.$ The distinctness of the edges implies that $j_a\neq j_b$ and $i_a\neq i_c.$ Then $R_b-R_a=\lambda(e_{j_a}-e_{j_b})\neq 0$ and $R_c-R_a=\lambda(e_{i_c}-e_{i_a})\neq 0.$ Now the edge $$R_c-R_b=\lambda(e_{i_c}-e_{i_a}-e_{j_a}+e_{j_b}),$$
    will be of the form $\lambda_{c,b}(e_i-e_j)$ only if one of the following holds:
    \[
    (i)\; i_c=i_a,\quad (ii)\; j_a=j_b,\quad(iii)\; i_c=j_a,\quad \text{and }(iv)\; i_a=j_b.
    \]
    The cases $(i)$ and $(ii)$ contradicts to the facts $R_c-R_a\neq 0$ and $R_b-R_a\neq 0$ respectively. Also the cases $(iii)$ and $(iv)$ are not possible as $i_a=i_b, j_a=j_c$ and $R_b-R_1, R_c-R_1\neq 0.$ Thus, in each case, we have contradictions. Therefore Claim 2 holds.

    Without loss of generality, assume that $i_a=i_k$ for all $a\neq 1.$
    Now combining Claim 1 and Claim 2, we have 
    \begin{equation}\label{n many rows}
        \{R_1,\ldots,R_n\}\subseteq
    \{\,R_1+\lambda(e_k-e_t):1\le t\le m\,\}.
    \end{equation}
    If $n>m,$ the observation \eqref{n many rows} ensures that there exist distinct $a,b\in \{1,\ldots ,n\}$ such that $R_a=R_b.$ Then the edge $R_a-R_b$ does not satisfy the hypothesis of the statement.
    Thus it follows that $n=m$ and this completes the proof.
\end{proof}

\begin{lemma}\label{T(e_i)}
    Let $T:\mathbb R^m \to \mathbb R^n$ be an injective linear map with $ m\ge 3$. Define the hypersurfaces $H_{p,q}:=\{(x_1,\dots,x_m)\in\mathbb R^m : x_p=x_q\},$ for $1\le p<q\le m$,
    and  
    $K_{u,v}:=\{(y_1,\dots,y_n)\in\mathbb R^n : y_u=y_v\},$ for $1\le u<v\le n$.
    Suppose
    \[
    T\Bigl(\mathbb R^m\setminus \bigcup_{1\le p<q\le m} H_{p,q}\Bigr)
    \;\subseteq\;
    \mathbb R^n\setminus \bigcup_{1\le u<v\le n} K_{u,v}.
    \]
    Then $m=n$ and there exist a permutation $\sigma$ on the set
    $\{1,2,\ldots,m\}$, a nonzero scalar $\lambda$, and scalars
    $b_i,\; i=1,2,\ldots,m,$ such that
    \[
    T(e_i)=b_i\mathbf 1_n-\lambda e_{\sigma(i)},\quad \text{for all }i=1,2,\ldots,m,
    \]
    where $e_i$ denotes the $i$-th standard basis vector and $\mathbf 1_n=\sum_{i=1}^n e_i=(1,\ldots,1)$.
\end{lemma}

\begin{proof}
    Let $A:=\bigcup_{p<q} H_{p,q}$ and $B:=\bigcup_{u<v} K_{u,v}$. The fact $T(\mathbb R^m \setminus A)\subseteq \mathbb R^n \setminus B$ implies 
    \[
    T^{-1}\Bigl(\bigcup_{1\le u<v\le n} K_{u,v}\Bigr)
    \subseteq
    \bigcup_{1\le p<q\le m} H_{ p,q}.
    \]
    
    Fix $u<v$. Since $K_{u,v}$ is a hyperplane, $T^{-1}(K_{u,v})$ is
    \begin{enumerate}
        \item either a hyperplane, if $T(\mathbb R^m)\nsubseteq K_{u,v}$,
        \item or  $\mathbb R^m$, if $T(\mathbb R^m)\subseteq K_{u,v}.$
    \end{enumerate}
      The possibility (2) is impossible, for otherwise $T(\mathbb R^m)\subseteq K_{u,v}\subseteq B$, contradicting the assumption $T(\mathbb R^m\setminus A)\subseteq \mathbb R^n\setminus B$ where $\mathbb R^m\setminus A\neq\emptyset$. Hence $T^{-1}(K_{u,v})$ is a hyperplane.
    
    Since $T^{-1}(K_{u,v})\subseteq T^{-1}(B)\subseteq A$, we have
    \[
    T^{-1}(K_{u,v})
    \subseteq
    \bigcup_{1\le p<q\le m} H_{p,q}.
    \]
    A hyperplane cannot be contained in a finite union of distinct hyperplanes unless it is contained in one of them. Therefore $T^{-1}(K_{u,v})\subseteq H_{p,q}$ for some pair $p<q$. Since $T^{-1}(K_{u,v}), H_{p,q}$ both are hyperplanes of $\mathbb R^m$, it follows that, for any $K_{u,v}$,
    \begin{equation}\label{T-1(Kuv)=Hpq}
        T^{-1}(K_{u,v})=H_{p,q}, \quad \text {for some } p<q.
    \end{equation}
    Define the linear functionals $\alpha_{p,q}: \mathbb R^m \to \mathbb R$ and $\beta_{u,v}: \mathbb R^n\to \mathbb R$ by
    \[
    \alpha_{p,q}(x_1,\ldots,x_m)=x_p-x_q\quad \text{and }\quad \beta_{u,v}(y_1,\ldots,y_n)=y_u-y_v.
    \]
    Now note that $H_{p,q}=\ker \alpha_{p,q}$ and $T^{-1}(K_{u,v})=\ker (\beta_{u,v}\circ T).$ Thus, using \eqref{T-1(Kuv)=Hpq}, we have $\ker \alpha_{p,q}=\ker (\beta_{u,v}\circ T).$ Since the linear functionals having the same kernal are scalar multiples of each other, it follows that there exists a \textit{nonzero} $\lambda_{u,v}$ such that
    \begin{equation}\label{beta T = c alpha}
        \beta_{u,v}\big(T(x)\big)= \lambda_{u,v}\alpha_{p,q}(x)\quad \text{for all }x\in \mathbb R^m.
    \end{equation}
    This shows that the $p$-th and $q$-th coordinates of $x$ are equal if and only if the $u$-th and $v$-th coordinates of $T(x)$ are also equal.
    
    Let $\{e_i\}_{i=1}^m$ denote the standard basis of $\mathbb R^m$. Now using \eqref{beta T = c alpha}, it follows that
    \[
    \beta_{u,v}\big(T(e_i)\big)=\begin{cases}
        \lambda_{u,v},&\text{if }i=p,\\
        -\lambda_{u,v},& \text{if }i=q,\\
        0, &\text{if }i\in \{1,2,\ldots, m\}\setminus\{p,q\}.
    \end{cases}
    \]
    
    Let $T$ correspond to the matrix $[T]:=[a_{ij}]=[R_1, R_2, \cdots , R_n]^T$ with respect to the standard matrix in both domain and codomain, where $R_i$ denotes the $i$-th row of $[T].$ Now observe that, for $x=(x_1, \ldots, x_m)\in \mathbb R^m$, we have $$\beta_{u,v}\big(T(x)\big)=\sum (a_{u,j}-a_{v,j})x_j= \lambda_{u,v}(x_p-x_q).$$
     This shows that for any $u< v$, there exist $p< q$ such that 
     \begin{equation}\label{row u-row v}
          R_u-R_v=\lambda_{u,v}(e_p-e_q), \quad \text{where }\lambda_{u,v}\in \mathbb R\setminus \{0\}.
     \end{equation}

    By Lemma~\ref{row - row is scaling and shift}, $m=n$ and there exist
    $R\in\mathbb R^m$, $\lambda\neq 0$, and
    $k\in\{1,\ldots,m\}$ such that
    \[
    \{R_1,\ldots,R_n\}\subseteq
    \{\,R+\lambda(e_k-e_t):1\le t\le m\,\}.
    \]
    Since $T$ is injective, $\operatorname{rank}([T])=m$. Hence all vectors
    $R+\lambda(e_k-e_t)$, $t=1,\ldots,m$, occur among the rows of $[T]$
    exactly once. Therefore, there exists a permutation $\sigma$ of the set
    $\{1,\ldots,m\}$ such that
    \[
    R_{\sigma(t)}=R+\lambda(e_k-e_t),\qquad t=1,\ldots,m.
    \]
    
    Let $C_j=T(e_j)$ denote the $j$-th column of $[T]$. Writing
    $R=(r_1,\ldots,r_m)$, we obtain
    \[
    (C_j)_{\sigma(t)}
    =(R_{\sigma(t)})_j
    =
    r_j+\lambda\delta_{k,j}-\lambda\delta_{t,j},
    \qquad t=1,\ldots,m.
    \]
    Therefore $C_j=(r_j+\lambda\delta_{k,j})\mathbf 1_m-\lambda e_{\sigma(j)},$
    where $\mathbf 1_m=(1,\ldots,1)^T$.
    Setting $b_j=r_j+\lambda\delta_{c,j},$ 
    we obtain $T(e_j)=b_j\mathbf1_n-\lambda e_{\sigma(j)},$ for all $j=1,\ldots,m.$ 
    This completes the proof.
\end{proof}

\begin{lemma}\label{extension}
    Let $\mathbf 1_k=(1,\ldots,1)=\sum_{i=1}^k e_i\in\mathbb R^k$, and let
    \[
    T:\mathbb R^m/\mathbb R\mathbf 1_m\to
    \mathbb R^n/\mathbb R\mathbf 1_n
    \]
    be an injective linear map. Then there exists an injective linear map $\widetilde T:\mathbb R^m\to\mathbb R^n$ such that
    $\widetilde T(\mathbb R\mathbf 1_m)
    \subseteq \mathbb R\mathbf 1_n$ and the induced map on the
    quotients is $T$.
    \end{lemma}
    
    \begin{proof}
    Let $L_m=\mathbb R\mathbf 1_m$ and
    $L_n=\mathbb R\mathbf 1_n$. Choose subspaces
    $V\subset \mathbb R^m$ and $W\subset \mathbb R^n$ such that
    $\mathbb R^m=L_m\oplus V$ and $\mathbb R^n=L_n\oplus W$.
    
    Let $p_m:\mathbb R^m\to \mathbb R^m/L_m$ and
    $p_n:\mathbb R^n\to \mathbb R^n/L_n$ denote the quotient maps.
    Since the restriction maps $p_m|_V$ and $p_n|_W$ are isomorphisms, the map
    \[
    L:=(p_n|_W)^{-1}\circ T\circ (p_m|_V):V\to W
    \]
    is well-defined. Moreover, $L$ is injective because $T$ is injective.
    
    Define $\widetilde T:\mathbb R^m\to\mathbb R^n$ by $\widetilde T(\alpha \mathbf 1_m+v)
    =
    \alpha \mathbf 1_n+L(v),
    \;
    \alpha\in\mathbb R,\ v\in V.$
    Then $\widetilde T$ is linear and
    $\widetilde T(L_m)\subseteq L_n$. Consequently, for $\alpha\in\mathbb R$ and $v\in V$,
    \[
    p_n(\widetilde T(\alpha \mathbf 1_m+v))
    =
    p_n(L(v))
    =
    T(p_m(v))
    =
    T(p_m(\alpha \mathbf 1_m+v)).
    \]
    Hence the map induced by $\widetilde T$ on the quotient spaces is
    precisely $T$.
    
    It remains to show that $\widetilde T$ is injective. Suppose
    $\widetilde T(\alpha \mathbf 1_m+v)=0$. Then $\alpha \mathbf 1_n+L(v)=0.$
    Since $\mathbb R^n=L_n\oplus W$, it follows that
    $\alpha \mathbf 1_n=0$ and $L(v)=0$. Thus $\alpha=0$ and, since $L$ is
    injective, $v=0$. Therefore
    $\alpha \mathbf 1_m+v=0$, proving that $\widetilde T$ is injective.
\end{proof}

\begin{remark}\label{over Q}
    Replacing $\mathbb{R}$ by $\mathbb{Q}$ throughout the statements and proofs of Lemmas~\ref{row - row is scaling and shift}, \ref{T(e_i)}, and \ref{extension} yields their corresponding versions over $\mathbb{Q}$. We shall use these versions in the following theorem.
\end{remark}

Now we are ready to prove our main result. 
We build on the idea of using the braid hyperplane arrangement in the study of cohomology homomorphisms from \cite{HH84}.

\begin{theorem}\label{main thm}
    Let $\mathbb{F}\in\{\mathbb{C},\mathbb{H}\}$, and let
    $f\colon \mathbb{F}G(\nu)\to \mathbb{F}G(\mu)$
    be a continuous map between $\mathbb{F}$-partial flag manifolds of the same dimension, where
    $\nu=(n_1,\ldots,n_r)$, $\mu=(m_1,\ldots,m_s)$, and $s\ge 3$. If $\nu$ is not a permutation of $\mu$, then $\deg(f)=0$.
\end{theorem}

\begin{proof}
    The proof is based on the fact that a map of nonzero degree induces a monomorphism in rational cohomology. Thus, if $\deg(f)\neq 0$, then the induced homomorphism
    $f^*:H^*(\mathbb FG(\mu);\mathbb Q)\to H^*(\mathbb FG(\nu);\mathbb Q)$
    is injective.

    First we prove for $\mathbb F=\mathbb C.$

    Let $\mathcal R$ and $\mathcal S$ denote the rational cohomology algebras of $\mathbb CG(\nu)$ and $\mathbb CG(\mu)$, respectively. Thus $\mathcal R=\bigoplus_{i=0}^d \mathcal R^i$ and $\mathcal S=\bigoplus_{i=0}^d \mathcal S^i$, where $\mathcal R^i=H^i(\mathbb CG(\nu);\mathbb Q), \;\mathcal S^i=H^i(\mathbb CG(\mu);\mathbb Q)$, and $d=\dim \mathbb CG(\nu)=\dim \mathbb CG(\mu)$.

      Let $x_i=c_1(\gamma_i)$ and $y_j=c_1(\eta_j)$, where $\gamma_i$ and $\eta_j$ denote the canonical complex vector bundles of ranks $n_i$ and $m_j$ over $\mathbb CG(\nu)$ and $\mathbb CG(\mu)$, respectively. Then
    \[
    \mathcal R^2 \;\cong\; \sum_{i=1}^r \mathbb Qx_i\Big/\mathbb Q\sum _{i=1}^rx_i
    \quad\text{and}\quad
    \mathcal S^2 \;\cong\;  \sum_{i=1}^s \mathbb Qy_i\Big/\mathbb Q\sum _{i=1}^sy_i.
    \] 

    Since $\dim \mathcal R^2=r-1$ and $\dim \mathcal S^2=s-1$, the map
    $f^*:\mathcal S^2\to \mathcal R^2$
    cannot be injective when $r<s$. Therefore, a nonzero degree map may exist only if $r\ge s$, and it suffices to consider this case.

    Assume that $f$ has nonzero degree. Then the induced map $f^*:\mathcal S\to \mathcal R$ is a monomorphism of graded $\mathbb Q$-algebras.

    For $p<q$, let us fix notation for the subspaces as:
    \[
    \mathcal R^2_{p,q}:=\Big\{\sum _{i=1}^ra_ix_i\in \mathcal R^2\mid a_p=a_q\Big\}\text{ and } \mathcal S^2_{p,q}:=\Big\{\sum_{i=1}^sb_iy_i\in \mathcal S^2\mid b_p=b_q\Big\}. 
    \]
    In view of Theorem~\ref{Kahler class description}, one deduce that the set of all elements in $\mathcal R^2, \mathcal S^2$ of maximal height $d$ are given by
    \[
    A:=\mathcal R^2\setminus \bigcup_{1\le p<q\le r}\mathcal R^2_{p,q}\quad\text{and}\quad B:=\mathcal S^2\setminus \bigcup_{1\le p<q\le s}\mathcal S^2_{p,q}, \quad \text{respectively}.
    \]
    Since a monomorphism preserves the height of the elements,  $f^*( B )\subseteq  A$. 
    Now by applying Lemma~\ref{T(e_i)} and Lemma~\ref{extension}, in view of Remark~\ref{over Q}, we have $r=s$ and there exist a permutation $\sigma$ of $\{1,2,\ldots, r\}$ and $\lambda\in \mathbb Q\setminus \{0\}$ such that
    \begin{equation}\label{f* permutes}
        f^*(y_i)=\lambda x_{\sigma(i)}, \qquad \text{for all } i=1,2,\ldots,r.
    \end{equation}

    Theorem~\ref{heights of deg 2 elements} implies that the elements belonging exclusively to $\mathcal R^2_{i,j}$ among $\{\mathcal R^2_{p,q}\}$ have height
    $d-n_i n_j$, while those belonging exclusively to $\mathcal S^2_{i,j}$ among $\{\mathcal S^2_{p,q}\}$ have height $d-m_i m_j$. Because $f^*$ preserves heights, \eqref{f* permutes} yields $f^*(\mathcal S^2_{i,j})=\mathcal R^2_{\sigma(i),\sigma(j)}$ and
    \[
    d-m_i m_j=d-n_{\sigma(i)}n_{\sigma(j)} \implies m_i m_j=n_{\sigma(i)}n_{\sigma(j)}\text{ for }1\le i<j\le r.
    \] It follows from this and Lemma~\ref{ninj=mimj} that $m_i=n_{\sigma(i)}$ for all $i=1,2,\ldots,r$. Hence, the defining sequences $\mu$ and $\nu$ agree up to the permutation $\sigma$, which contradicts the assumption that the partial flag manifolds $\mathbb CG(\mu)$ and $\mathbb CG(\nu)$ are distinct. Therefore, we conclude that $\deg(f)=0$.

    For the case $\mathbb{F}=\mathbb{H}$, the argument is completely analogous. Indeed, the rational cohomology algebra of a quaternionic partial flag manifold is obtained from that of the corresponding complex partial flag manifold by doubling the degrees of the generators. Thus, the proof for the quaternionic case follows by the same reasoning.
\end{proof}

\begin{remark}
    Theorem~\ref{main thm} has an immediate algebraic-geometric consequence. Let $X$ and $Y$ be complex partial flag manifolds that are not Grassmannians. If there exists a finite surjective morphism $f:X\to Y$, then $f$ necessarily has nonzero degree. Therefore, Theorem~\ref{main thm} implies that $X$ and $Y$ are isomorphic as complex varieties. This is the analogue of Proposition~6 of \cite{PS89}, where the result is established for complex Grassmannians.
\end{remark}

Now we study the existence of nonzero-degree maps $f:\mathbb FG(\nu)\to \mathbb FG(\mu)$ when $\mu$ has length $2$ and $\nu$ has length at least $3$.

\begin{proposition}\label{prop r>2 s=2}
    Let $\mathbb{F}\in\{\mathbb{C},\mathbb{H}\}$, and let
    $f\colon \mathbb{F}G(\nu)\to \mathbb{F}G(\mu)$
    be a continuous map between $\mathbb{F}$-partial flag manifolds of the same dimension, where
    $\nu=(n_1,\ldots,n_r), r\ge 3$, and $\mu=(m_1, m_2)$. Then the Brouwer degree of $f$ is zero.  
\end{proposition}

\begin{proof}
    We adopt the construction and notation introduced in the proof of Theorem~\ref{main thm}. 
    
    Suppose for the sake of contradiction that $\deg(f) \neq 0$. Since $f$ has non-zero degree, the induced homomorphism $f^*$ on the rational cohomology algebra is injective, and hence preserves the heights of cohomology classes. Consequently, $f^*$ restricts to an inclusion on the complements:
    \begin{equation*}
        f^*\big( \mathcal{S}^2 \setminus \bigcup_{1 \le p < q \le 2} \mathcal{S}^2_{p,q} \big) \subseteq \mathcal{R}^2 \setminus \bigcup_{1 \le p < q \le r} \mathcal{R}^2_{p,q}.
    \end{equation*}
    Taking the complements yields the reverse inclusion:
    \begin{equation}\label{contains the hyper planes}
        f^*\big( \bigcup_{1 \le p < q \le 2} \mathcal{S}^2_{p,q} \big) \supseteq \bigcup_{1 \le p < q \le r} \mathcal{R}^2_{p,q}.
    \end{equation}
    
    Observe that $\bigcup_{1 \le p < q \le 2} \mathcal{S}^2_{p,q}=\mathcal S^2_{1,2} = \{0\}$, the zero vector space, since $y_1+y_2=0$ in $\mathcal S^2$. On the other hand, since $r\ge 3$, the union $\bigcup_{1 \le p < q \le r} \mathcal{R}^2_{p,q} \neq \{0\}$ as it consists hyperplanes of dimension at least $1$. This contradicts inclusion~\eqref{contains the hyper planes}, as a linear transformation cannot map $\{0\}$ onto a non-zero space. 
    
    We conclude that $f$ must have degree zero.
\end{proof}

Note that Theorem~\ref{main thm} and Proposition~\ref{prop r>2 s=2} settle the cases of maps between partial flag manifolds in which either the domain or the codomain is not a Grassmannian, that is, $(r,s)\neq(2,2)$. Thus, the only remaining case not covered here is $r=s=2$. This case has been studied extensively in \cite{RS97,SS09}. For completeness, we recall the relevant results and summarize all the cases in the Table~\ref{tab:cases_classification}. The table provides a concise overview of the developments around the problem.

Let $f \colon \mathbb{F}\mathrm{G}(\nu) \to \mathbb{F}\mathrm{G}(\mu)$ be a continuous map between two $\mathbb{F}$-partial flag manifolds of equal dimension, where $\nu=(n_1,\ldots, n_r)$ and $\mu=(m_1,\ldots,m_s)$ are \textit{weakly increasing} distinct sequences of natural numbers and $\mathbb F\in \{\mathbb C, \mathbb H\}$. Define the quantity
\[
\mathcal Q:=(n_1^2-1)\cdots(n_r^2-1)(m_1^2-1)\cdots(m_s^2-1).
\]
Then the degree outcomes for $f$ across all cases are summarized in Table~\ref{tab:cases_classification}.

\begin{table}[htbp]
\centering
\small
\renewcommand{\arraystretch}{1.4} 
\begin{tabularx}{\textwidth}{c c X l l}
\toprule
\textbf{$\boldsymbol{s}$} & \textbf{$\boldsymbol{r}$} & \textbf{Assumptions} & \textbf{$\boldsymbol{\deg(f)}$} & \textbf{Reference/Status} \\
\midrule
\multirow{5}{*}{$2$} & \multirow{4}{*}{$2$} 
  & $1 \le n_1 < m_1$ & Zero & \cite[Theorem~2]{RS97} \\
    \cmidrule(lr){3-5}
  & & $2 \le m_1 < n_1$ and $\mathcal Q$ is not a perfect square & Zero & \cite[Theorem~1.2]{SS09} \\
    \cmidrule(lr){3-5}
  & & $1=m_1<n_1$  \textcolor{gray}{$\big(\text{e.g. } \nu=(2,3), \mu=(1,6)\big)$} & Zero/nonzero & Construction~\ref{nonzero deg map to Pn}\\
  \cmidrule(lr){3-5} 
  & & $1\le m_1<n_1$ and $\mathcal Q$ is a perfect square $\quad$ \textcolor{gray}{$\big(\text{e.g. } \nu=(69,1121), \mu=(23,3363)\big)$} & Unknown & Remains Open \\
\cmidrule(lr){2-5}
 & $\ge3$ & --- & Zero & Proposition~\ref{prop r>2 s=2} \\
\midrule
$\ge 3$ & --- & --- & Zero & Theorem~\ref{main thm} \\
\bottomrule
\end{tabularx}
\caption{Degree outcomes across all cases}
\label{tab:cases_classification}
\end{table}

J. Oesterlé found that when $(n_1,n_2)=(69,1121)$, $(m_1,m_2)=(23,3363)$, then $Q$ is a perfect square; see \cite[Example~4.3.5]{Sar10}. This shows that the case $r=s=2$, $1\le m_1<n_1$, with $\mathcal Q$ a perfect square, is nonempty and remains unresolved. The conjecture of \cite{RS97} predicts that every map in this case has degree zero.

We next discuss a standard construction of nonzero-degree maps from complex Grassmannians to complex projective spaces.

\begin{construction}[Nonzero degree maps to projective spaces]\label{nonzero deg map to Pn}

Recall the description of the cohomology algebra of the complex Grassmannian $\mathbb CG(n_1,n_2)$ from \eqref{cohom c_i}. Its second integral cohomology group $H^2(\mathbb CG(n_1,n_2); \mathbb{Z}) \cong \mathbb{Z}$ is generated by the first Chern class $c_1(\gamma_1)$, where $\gamma_1$ denotes the canonical complex vector bundle of rank $n_1$ over $\mathbb CG(n_1,n_2)$.

Since complex line bundles over a CW complex are classified by its second cohomology group, there exists a complex line bundle $\xi$ over $\mathbb CG(n_1,n_2)$ whose first Chern class $$c_1(\xi)=c_1(\gamma_1) \in H^2(\mathbb CG(n_1,n_2); \mathbb{Z}).$$ A natural choice for such a line bundle is the determinant line bundle $\Lambda^{n_1}\gamma_1$.

Since complex line bundles (up to isomorphism) over a paracompact Hausdorff space $X$ are classified by homotopy classes of maps from $X$ into the infinite complex projective space $\mathbb CP^\infty$, the bundle $\xi$ is classified by a continuous map 
\begin{equation*}
    f \colon \mathbb CG(n_1,n_2) \to \mathbb CP^\infty.
\end{equation*}

By the cellular approximation theorem, $f$ is homotopic to a map 
\[
g \colon \mathbb CG(n_1,n_2) \to \mathbb CP^d,\quad \text{where } d=n_1n_2=\dim_{\mathbb C} \mathbb CG(n_1,n_2).
\]
Thus, $g$ also classifies $\xi$. That is, the canonical line bundle $\zeta$ over $\mathbb CP^d$ pulls back to $\xi$ under $g$ (up to isomorphism). By the naturality of Chern classes, we obtain
\begin{equation}\label{gen map to gen}
g^*(c_1(\zeta)) = c_1(\xi).
\end{equation}

By Theorem~\ref{Kahler class description}, the class
$c_1(\xi)=c_1(\gamma_1)$ is a K\"ahler class. Hence
$c_1(\xi)^d\neq 0$ in
$H^{2d}(\mathbb CG(n_1,n_2);\mathbb Z)$, where
$d=n_1n_2$. Thus, we have 
\[
c_1(\xi)^d=N\vartheta, \quad \text{for some }N\in\mathbb Z\setminus\{0\},
\]
where $\vartheta$ is a generator of
$H^{2d}(\mathbb CG(n_1,n_2);\mathbb Z)$. The value of $N$
up to sign is given by
\[
|N|
=
\left|
\left\langle
c_1(\xi)^d,
[\mathbb CG(n_1,n_2)]
\right\rangle
\right|
=
\frac{
(n_1n_2)!\,1!\,2!\cdots(n_1-1)!
}{
n_2!\,(n_2+1)!\cdots(n_1+n_2-1)!
},
\]
where $[\mathbb CG(n_1,n_2)]$ denotes the fundamental homology class of $\mathbb CG(n_1,n_2)$.
For details, see \cite[Lemma~2.3.1]{Sar10} and
\cite[\S~14.7]{Ful98}.

The $d$-th power $c_1(\zeta)^d$ generates $H^{2d}(\mathbb CP^d; \mathbb{Z}) \cong \mathbb{Z}$. Now applying \eqref{gen map to gen}, we have
\[
g^*\left(c_1(\zeta)^d\right) = c_1(\xi)^d=N\vartheta.
\]
This shows that the map $g$ has degree $\pm N$ (depending on the choices of orientation for $\mathbb CG(n_1,n_2)$ and $\mathbb CP^d$). Since $f$ is homotopic to $g$, the degree of $f$ is $\pm N\in \mathbb Z\setminus \{0\}$.

Null-homotopic maps provide a class of degree-zero maps. On the other hand, using the above construction,  considering line bundles corresponding to positive integral multiples of $c_1(\gamma_1)$ yields maps from
$\mathbb CG(n_1,n_2)$ to $\mathbb CP^d$ of arbitrarily large degree.
\end{construction}

\section{Cohomological rigidity of partial flag manifolds}

In this section, we apply Theorem~\ref{main thm} to establish the cohomological rigidity of complex and quaternionic partial flag manifolds.

We begin by recalling the following result that will be used in the proof.

\begin{theorem}[Theoerem 1.2 \cite{GH81}]\label{induced by topological map}
    Let $X$ and $Y$ be formal nilpotent finite CW complexes and $\varphi\colon H^*(Y;\mathbb{Q}) \to H^*(X;\mathbb{Q})$ be a graded ring homomorphism.  Then there exists an isomorphism $f^*$ of $H^*(X;\mathbb Q)$ induced by a self-map $f$ on $X$  such that $f^* \circ \varphi$ is induced by a map $g\colon X \to Y$.
\end{theorem}

We now proceed to prove the cohomological rigidity result for partial flag manifolds.

\begin{theorem}\label{prop:cohomological-rigidity}
    Let $\mathbb{F}=\mathbb{C}$ or $\mathbb{H}$, and let
    $M=\mathbb{F}G(m_1,\ldots,m_r)$ and
    $N=\mathbb{F}G(n_1,\ldots,n_s)$
    be partial flag manifolds. If
    $H^*(M;\mathbb{Q})\cong H^*(N;\mathbb{Q})$
    as graded $\mathbb{Q}$-algebras, then the partial flag manifolds $M$ and $N$ are homeomorphic.
\end{theorem}

\begin{proof}
    Let $d=\dim_{\mathbb{R}}\mathbb{F}$. Since
    $H^d(M;\mathbb{Q})\cong \mathbb{Q}^{\,r-1}$ and
    $H^d(N;\mathbb{Q})\cong \mathbb{Q}^{\,s-1}$,
    the isomorphism
    $H^d(M;\mathbb{Q})\cong H^d(N;\mathbb{Q})$
    implies that $r=s$. Since the partial flag manifolds $M,N$ are simply connected, by Theorem~\ref{induced by topological map}, there exist an isomorphism $H^*(M;\mathbb Q)\cong H^*(N;\mathbb Q)$ which is induced by a continuous map $f:M\to N.$ Here, $\deg(f)\neq 0.$ 

    If $r\ge 3$, then Theorem~\ref{main thm} implies that the sequences
    $(m_1,\ldots,m_r)$ and $(n_1,\ldots,n_r)$
    coincide up to permutation. Hence $M$ and $N$ are homeomorphic.

    Suppose now that $r=2$. Then
    $M=\mathbb{F}G(m_1,m_2)$ and
    $N=\mathbb{F}G(n_1,n_2)$.
    Since $H^*(M;\mathbb{Q})\cong H^*(N;\mathbb{Q})$ as graded $\mathbb{Q}$-algebras, their top nonvanishing cohomology groups occur in the same degree. Hence $M$ and $N$ have the same dimension
    $m_1m_2=n_1n_2$.

    Without loss of generality, assume that
    $m_1\le m_2$ and $n_1\le n_2$.
    If $m_1=n_1$, then the equality
    $m_1m_2=n_1n_2$ immediately implies that
    $(m_1,m_2)=(n_1,n_2)$, and consequently $M$ and $N$ are homeomorphic.

    Assume, for a contradiction, that $m_1<n_1$.
    Since $m_1m_2=n_1n_2$ and $n_1\le n_2$, we must have
    $m_1<n_1\le n_2<m_2$.

    First consider the case $\mathbb{F}=\mathbb{C}$.
    The rational cohomology ring of $M$ is generated by the Chern classes
    $c_1,\ldots,c_{m_1}$ of the canonical complex vector bundle of rank $m_1$, and its defining relations first occur in degree $2(m_2+1)$; see Theorem~\ref{cohom c_i}. Since $m_2>m_1$, there are no relations in degrees up to $2(m_1+1)$. Consequently,
    $\dim H^{2(m_1+1)}(M;\mathbb{Q})$ equals the number of monomials of total degree $m_1+1$ in the variables $c_1,\ldots,c_{m_1}$, where $\deg c_i=2i$.
    
    Similarly, using Theorem~\ref{cohom c_i}, $H^*(N;\mathbb{Q})$ is generated by the Chern classes
    $d_1,\ldots,d_{n_1}$ of the canonical complex vector bundle of rank $n_1$, and since $n_1>m_1$, the additional generator
    $d_{m_1+1}$ is available. As there are again no relations in degree $2(m_1+1)$, the dimension of $H^{2(m_1+1)}(N;\mathbb{Q})$ is the number of monomials of total degree $m_1+1$ in $d_1,\ldots,d_{m_1},d_{m_1+1}$. The monomial $d_{m_1+1}$ itself has degree $2(m_1+1)$ and cannot be expressed as a polynomial in
    $d_1,\ldots,d_{m_1}$. Hence
    $$\dim H^{2(m_1+1)}(N;\mathbb{Q})
    >
    \dim H^{2(m_1+1)}(M;\mathbb{Q}),$$
    contradicting the assumption that
    $H^*(M;\mathbb{Q})\cong H^*(N;\mathbb{Q})$.
    
    Thus $m_1<n_1$ cannot occur. Therefore
    $(m_1,m_2)=(n_1,n_2)$, and hence the Grassmannians $M$ and $N$ are homeomorphic.

    The proof for $\mathbb F=\mathbb H$ is analogous, the only difference being that the generators occur in twice the degree. We therefore omit the details.
\end{proof}

\begin{remark}\label{rank comparison}
    In the proof of Theorem~\ref{prop:cohomological-rigidity}, the case of Grassmannians (i.e., $r=s=2$) is treated differently from that of general partial flag manifolds. There, the argument relies only on comparing the ranks of the cohomology groups in suitable degrees. In particular, if two complex (respectively, quaternionic) Grassmannians are distinct, then their integral cohomology groups differ in some degree. Thus, the proof does not require the ring structure of the cohomology algebra.
    
    Moreover, since the integral cohomology of complex and quaternionic Grassmannians is torsion free, the Universal Coefficient Theorem gives an isomorphism
    \[
    H^i(\mathbb FG(\nu);\mathbb{Z})\cong \operatorname{Hom}\bigl(H_i(\mathbb FG(\nu);\mathbb{Z}),\mathbb{Z}\bigr)\text{ for all }i,\quad \text{where }\mathbb F\in \{\mathbb C, \mathbb H\}.
    \]
    Consequently, the integral homology groups determine the integral (or, rational) cohomology groups, and vice versa. It follows that our argument actually proves a stronger statement: complex (respectively, quaternionic) Grassmannians are \emph{homologically rigid}. Namely, if two such Grassmannians have isomorphic integral homology groups in every degree, then they are homeomorphic.
    \qed
\end{remark}

\begin{remark}\label{homological rigidity of almost complete flags}
    Almost complete flag manifolds also exhibit homological rigidity. To see this, let
    $\mathbb{F}G(1,\ldots,1,n_r)$ and $\mathbb{F}G(1,\ldots,1,m_s)$ be two almost complete flag manifolds, where $\mathbb{F}\in\{\mathbb{C},\mathbb{H}\}$, whose cohomology groups are isomorphic. Since their top nonvanishing cohomology groups occur in the same degree, the manifolds have the same real dimension.
    
    Next, observe that the Betti number $b_d$, where $d=\dim_{\mathbb{R}}\mathbb{F}$, is exactly the number of $1$'s appearing in the defining sequence. Thus the defining sequences contain the same number of $1$'s. As the dimensions are equal, it follows that the remaining entries must also coincide, that is, $n_r=m_s$. Therefore the defining sequences are identical. Consequently, almost complete flag manifolds are determined, up to homeomorphism, by their cohomology groups, or equivalently, homology groups.\qed
\end{remark}

We conclude with the following conjecture on the homological rigidity.

\begin{conjecture}\label{hom rig}
    Let $\mathbb{F}\in{\mathbb{C},\mathbb{H}}$, and let $\mathbb{F}G(\nu)$ and $\mathbb{F}G(\mu)$ be two partial flag manifolds. If $H_i\bigl(\mathbb{F}G(\nu);\mathbb{Z}\bigr)
    \cong
    H_i\bigl(\mathbb{F}G(\mu);\mathbb{Z}\bigr)$ for every $i$, then $\mathbb{F}G(\nu)$ and $\mathbb{F}G(\mu)$ are homeomorphic.
\end{conjecture}

Remark~\ref{rank comparison} and Remark~\ref{homological rigidity of almost complete flags} demonstrates that homological rigidity holds for complex and quaternionic Grassmannians as well as almost complete flag manifolds. These observations provide supporting evidence for Conjecture~\ref{hom rig}.

\subsection*{Acknowledgements}

The author sincerely thanks Prof.~Parameswaran Sankaran for his valuable comments and suggestions on an earlier version of this manuscript, and gratefully acknowledges the postdoctoral fellowship from IIT Kanpur.


\bibliographystyle{alpha}
\bibliography{references}

\end{document}